\documentclass[11pt]{amsart}

\theoremstyle{plain}
\newtheorem{theorem}{Theorem}[section]
\newtheorem{proposition}[theorem]{Proposition}
\newtheorem{lemma}[theorem]{Lemma}

\theoremstyle{definition}
\newtheorem{definition}[theorem]{Definition}
\newtheorem{example}[theorem]{Example}
\newtheorem{remark}[theorem]{Remark}
\theoremstyle{remark}

\newcommand{\enne}{\mathbb{N}}

\newcommand{\elle}[1]{L^{#1}(\Omega)}

\newcommand{\re}{\mathbb{R}} 
\newcommand{\R}{\mathbb{R}}

\def\al{\alpha} 
\def\rn{\mathbb{R}^{N}} 
 
\def\D{\nabla} 
 
\def\vp{\varphi} 
\def\rife#1{(\ref{#1})} 
 
\def\eps{\varepsilon}

\def\dys{\displaystyle}

\def\Om{\Omega}

\def\bc{\begin{cases}} 
\def\ec{\end{cases}} 
\def\be{\begin{equation}} 
\def\ee{\end{equation}} 
 
\def\ga{\gamma}

\def\ga{\gamma}
\def\la{\lambda}

\def\t1pn{\mathcal T^{1,p}_{\rm loc} ((0,T)\times\rn)}

\def\vare{\varepsilon}

\keywords{Nonlinear parabolic equations, Singular parabolic equations, Degenerate parabolic equations}
\title[Degenerate parabolic equations with singular terms]{Existence of solutions for degenerate parabolic equations with singular terms}

\author[A. Dall'Aglio]{Andrea Dall'Aglio}\address{Andrea Dall'Aglio, Dipartimento di Matematica ``G. Castelnuovo'',
 ``Sapienza" Universit\`a di Roma, Piazzale Aldo Moro 5, 00185 Roma, Italy.}\email{dallaglio@mat.uniroma1.it}

\author[L. Orsina]{Luigi Orsina}\address{Luigi Orsina, Dipartimento di Matematica ``G. Castelnuovo'',
 ``Sapienza" Universit\`a di Roma, Piazzale Aldo Moro 5, 00185 Roma, Italy.}\email{orsina@mat.uniroma1.it}

\author[F. Petitta]{Francesco Petitta}\address{Francesco Petitta, Dipartimento di Scienze di Base e Applicate
per l' Ingegneria,
 ``Sa\-pien\-za" Universit\`a di Roma, Via Scarpa 16, 00161 Roma, Italy.}\email{francesco.petitta@sbai.uniroma1.it}

\begin{document}

\begin{abstract}
In this paper we deal with parabolic problems whose simplest model is 
$$
\left\{
\begin{array}{cl}
u'- \Delta_{p} u + B\frac{|\D u|^p}{u} = 0 \quad &\mbox{in $(0,T) \times \Omega$,} \\[1.5 ex]
u(0,x)= u_0 (x) &\mbox{in $\Omega$,} \\[1.5 ex]
u(t,x)=0 &\mbox{on $(0,T) \times \partial\Omega$,}\quad 
\end{array}\right.
$$
where $T>0$, $N\geq 2$, $p>1$, $B > 0$, and $u_{0}$ is a positive function in $L^{\infty}(\Omega)$ bounded away from zero.
\end{abstract}

\maketitle



\section{Introduction}
In this paper we are concerned with homogeneous nonlinear singular initial-boundary value problems whose model is 
$$
\dys u'- \Delta_{p} u + B\frac{|\D u|^p}{u} = 0\,,
$$
where $\Omega$ is a bounded open set of $\rn$, $N\geq 2$, $p>1$, $B\in \re^{+}$ and $\Delta_{p} u={\rm div} (|\nabla u|^{p-2}\nabla u)$ is the usual $p$-laplacian. Here and below we use the simplified notation $u'$ in order to indicate the time derivative of $u$ with respect to $t$.

Singular problems of this type have been largely studied in the past also for their connection with the theory of non-Newtonian fluids and heat conduction in electrically active materials (see for instance \cite{nc,kc} and references therein). 

From the mathematical point of view, the non-homogeneous elliptic case has been considered in a series of papers in the last few years. Consider the equation 
\begin{equation}\label{a6}
- \Delta u + B\frac{|\D u|^2}{u^{\gamma}} = f\,,
\end{equation}
equipped with homogeneous Dirichlet boundary conditions. Here $\gamma>0$ and $f$ is a nonnegative (not identically zero) function in $L^{1}(\Omega)$. The problem is obviously singular as we ask the solution to vanish at the boundary of $\Omega$. In \cite{a6} the existence of a finite energy (i.e., in $H^{1}_{0}(\Omega)$) solution to problem \rife{a6} has been proved if $\gamma<2$ and for data $f$ locally bounded away from zero.

The case of a possibly degenerate datum $f$ has been also considered. 
If $\gamma<1$ the existence of a solution in $H^{1}_{0}(\Omega)$ was proved in \cite{b1} for general nonnegative (not identically zero) data, while the case $\gamma=1$ was faced in \cite{m1} provided $B$ was small enough (we also mention \cite{gp, pv, ACM, ABLP} and the references therein for a quite complete account on the subject).
Problems as in \rife{a6} with possibly changing-sign data have also been considered in \cite{gps} in the case $\gamma<1$ (see also \cite{gps2} for further considerations concerning the strongly singular case).

\medskip 
In the evolutive case, problems as 
\begin{equation}\label{evol}
\left\{
\begin{array}{cl}
\dys u'- \Delta_{p} u + B\frac{|\D u|^p}{u^{\gamma}} = f \quad &\mbox{in $(0,T) \times \Omega$,} \\[1.5 ex]
u(0,x)= u_0 (x) &\mbox{in $\Omega$,} \\[1.5 ex]
u(t,x)=0 &\mbox{on $(0,T) \times \partial\Omega$,}\quad 
\end{array}\right.
\end{equation}
have been considered in the case $p=2$ and $\gamma<1$ (see \cite{mape}). If $\gamma=1$ singular problems as \rife{evol} have been considered in \cite{xy,xy1} for smooth strictly positive data, while degenerate problems (i.e. $p>2$) were studied in \cite{zw} in the one dimensional case. 

We also would like to stress that, in the degenerate case $p>2$ and if $B>p-1$, to find a solution for the model problem \rife{evol} is formally equivalent, through a standard Cole-Hopf transformation, to find a {\sl large solution} to the doubly nonlinear problem
$$
w'=\Delta_{p} w^{m}, \ \ \ m<\frac{1}{p-1}, 
$$
that is a solution that blows up at the boundary of $\Omega$ (see \cite{cv, lp, biv, mp} for further considerations on this fact).

The aim of this paper is to study existence and nonexistence of solutions for a general class of singular homogeneous (i.e. $f\equiv 0$) parabolic problems as \rife{evol} in the limit case $\gamma=1$. We will mainly be concerned with the case $p\geq 2$. 

\medskip
The paper is structured as follows: in the next section we set the main assumptions, we state our main result, and we introduce some preliminary tools. Section \ref{3} is devoted to prove existence of a solution in the degenerate case $p>2$. In Section \ref{4} we prove the existence of a solution in the case $p=2$ provided the size of the lower order term is small (e.g., $B<1$ in \rife{evol}), while in Section \ref{5} we will prove a nonexistence result in the complementary case $B\geq 1$. In the last section of the paper we will also give an account on the singular case $p<2$ by showing some finite time extinction results, and by discussing some explicit examples of evolution.

\subsection*{Notations.} From now on, we will set $Q = (0,T) \times \Omega$ and $\Gamma = (0,T) \times \partial\Omega$.
If not otherwise specified, we will denote by $C$ several constants whose value may change from line to line and, sometimes, on the same line. These values will only depend on the data (for instance $C$ can depend on $N$, $\Omega$, $T$, $B$, $\alpha$, and $\beta$) but they will never depend on the indexes of the sequences we will often introduce. For the sake of simplicity we will often use the simplified notation
$$
\int_{Q} f\doteq\int_{Q} f(t,x)\ dtdx\,,
$$
when referring to integrals when no ambiguity on the variable of integration is possible.

For fixed $k>0$ we will made use of the truncation functions $T_{k}$ and $G_{k}$ defined as
$$
T_k(s) = \max (-k,\min (s,k))\,, 
$$
and
$$
G_k(s)= s - T_{k}(s) = (|s|-k)^+ \operatorname{sign}(s)\,. 
$$

\section{Setting of the problem and preliminary results}

We consider a parabolic differential problem of the form
\begin{equation}\label{pro}
\left\{
\begin{array}{cl}
\dys u'- {\rm div}\,a(t,x,\nabla u) + H(t, x,u,\nabla u) = 0 \quad & \mbox{in $Q$,} \\[1.5 ex]
u(0,x)= u_0 (x) &\mbox{in $\Omega$,} \\[1.5 ex]
u(t,x)=0 &\mbox{on $\Gamma$,}\quad 
\end{array}\right.
\end{equation}
where 
\begin{itemize}
\item $\Om$ is a bounded open set in $\R^N$, with $N\geq 2$, $T$ is a positive number;
\item $a(t,x,\xi)\ :\ (0,T)\times \Om \times \R^N \to \R^N$ is a Carath\'eodory vector-valued function such that
\begin{equation}\label{oper1}
 a(t,x,\xi)\cdot \xi \geq \al |\xi|^p\,,
\end{equation}
\begin{equation}\label{oper2}
 |a(t,x,\xi)| \leq \beta |\xi|^{p-1}\,,
\end{equation}
\begin{equation}\label{oper3}
 (a(t,x,\xi)- a(t,x,\eta))\cdot (\xi-\eta)>0 \,,
\end{equation}
for a.e. $(t,x) \in (0,T)\times \Om$, for every $\xi, \eta \in \R^N$, with $\xi\not=\eta$, where $\al, \beta$ are positive constants and $p>1$;
\item $H(t,x,s, \xi)\ :\ (0,T)\times \Om \times \R_+\times \R^N \to \R$ is a Carath\'eodory function such that
\begin{equation}\label{operH1}
 0\leq H(t,x,s, \xi)\leq B\frac{|\xi|^p}{s}\,,
\end{equation}
for a.e. $(t,x) \in (0,T)\times \Om$, for every $s>0$, $\xi\in \R^N$, where $B$ is a positive constant;
\item $u_{0}(x)$ is a function in $L^{\infty}(\Omega)$ such that $u_{0}\geq c>0$ almost everywhere on $\Omega$.
\end{itemize}
In \eqref{pro}, we denote by $u'$ the partial derivative with respect to time, while $\nabla u$ stands for the gradient with respect to the space variable $x$.

Consider problem \rife{pro}. Here is the meaning of weak solution for such a problem.

\begin{definition}\label{defin}
A weak solution to problem \rife{pro} is a function $u$ in $L^{p}(0,T; W^{1,p}_{0}(\Omega))\cap C(0,T;L^{1}(\Omega))$ such that for every $\omega\subset\subset \Omega$ there exists $c_{\omega}$ such that $u\geq c_{\omega}>0$ in $(0,T) \times \omega$; furthermore, we have that 
$$
-\int_Q u\varphi' +\int_{Q}a(t,x,\nabla u)\cdot \nabla \varphi +\int_{Q} H(t, x, u, \nabla u)\,\varphi=\int_{\Omega} u_{0}\varphi(0)\,,
$$
for every $\varphi\in C^{1}_{\rm c}([0,T)\times\Omega)$, that is, for every $C^1$ function which vanishes in a neighborhood of $\{T\}\times \Om$ and of $ (0,T)\times \partial \Om$. 
\end{definition}

Note that under assumption \eqref{operH1}, the function $H(t,x,u,\D u)$ belongs to $L^{1}_{{\rm loc}}(Q)$ thanks to the property of local positivity required on $u$.

\medskip

Our main result is the following:
\begin{theorem}\label{pm2}\sl
If $p>2$, there exists a weak solution to problem \eqref{pro}. Moreover, if $p=2$, there exists a solution if the constant $B$ appearing in \eqref{operH1} satisfies $B<\al$.
\end{theorem}

\begin{remark}
Assumption $B<\al$ is, in some sense, optimal if $p=2$: this will be the content of Proposition \ref{none} below, in which we show how approximating problems may degenerate if $B\geq \al$. Assumption $p\geq 2$ is also needed in this context since, as we will see in Section \ref{6}, finite time extinction can occur if $p<2$, so that Definition \ref{defin} should be suitably modified. 
\end{remark}

Our strategy in order to prove Theorem \ref{pm2} will rely on an approximation argument. The next subsection will introduce our approximating problems.

\subsection{The approximating problems}

We consider the approximating problems
 \begin{equation}\label{pron}
\left\{
\begin{array}{cl}
\dys 
 u_n'- {\rm div}\,a(t,x,\nabla u_n) + H(t, x,u_n,\nabla u_n) = 0 &\mbox{in $Q$,} \\[1.5 ex]
\dys u_n(0,x)= u_0 (x)+\frac1n &\mbox{in $\Omega$,} \\[1.5 ex]
\dys u_n(t,x)=\frac1n &\mbox{on $\Gamma$.} 
\end{array}\right.
\end{equation}
A weak solution to this problem is a function $u_n$ such that $u_n\geq\frac1n$ a.e.\ in $Q$, $u_{n}-1/n \in L^{p}(0,T; W^{1,p}_{0}(\Omega))\cap C([0,T];L^{1}(\Omega))$ and $u_{n}'\in L^{1}(Q)+L^{p'}(0,T;W^{-1,p'}(\Omega))$, and such that
 \begin{equation}\label{pronweak}
\int_{0}^{T}\langle u_{n}',v\rangle +\int_{Q}a(t,x,\nabla u_n)\cdot\nabla v +\int_{Q} H(t, x,u_n,\nabla u_n)\,v=0\,,
 \end{equation}
for any $v\in L^{p}(0,T;W^{1,p}_{0}(\Omega))\cap L^{\infty}(Q)$. 
A nonnegative weak solution $u_{n}$ to problem \rife{pron} does exist. In fact, 
problem \eqref{pron} is equivalent to 
 \begin{equation}\label{pron3}
\left\{
\begin{array}{cl}
\dys 
 v_n'- {\rm div}\,a(t,x,\nabla v_n) + H(t, x,v_n+\frac1n,\nabla v_n) = 0 \quad &\mbox{in $Q$,} \\[1.5 ex]
\dys v_n(0,x)= u_0 (x) &\mbox{in $\Omega$,} \\[1.5 ex]
\dys v_n(t,x)=0 &\mbox{on $\Gamma$,} 
\end{array}\right.
\end{equation}
where $v_n = u_n-\frac1n$.

To prove that a solution of \eqref{pron3} exists, we first extend $H(t,x,s,\xi)$ to zero for $s\leq 0$, and, for $\eps>0$, we consider the problem 
 \begin{equation}\label{pron2eps}
\left\{
\begin{array}{cl}
\dys 
v_{n,\eps}'- {\rm div}\,a(t,x,\nabla v_{n,\eps}) + \frac{T_\eps(v_{n,\eps})}{\eps} H(t, x,v_{n,\eps}+\frac1n,\nabla v_{n,\eps}) = 0 &\mbox{in $Q$,} \\[1.5 ex]
\dys v_{n,\eps}(0,x)= u_0 (x) &\mbox{in $\Omega$,} \\[1.5 ex]
\dys v_{n,\eps}(t,x)=0 &\mbox{on $\Gamma$.} 
\end{array}\right.
\end{equation}
A nonnegative solution $v_{n,\eps}$ to problem \eqref{pron2eps} exists by the results proven in \cite{do}.
Then, if we take a sequence of values $\eps\downarrow 0$, one can prove that the sequence $\{v_{n,\eps}\}_{\eps}$ converges strongly in $L^p(0,T;W^{1,p}_0(\Om))$ to some function $v_n$.
Then one can pass to the limit for $\eps\downarrow 0$ in the first two terms of \eqref{pron2eps}, in the sense of distributions. As far as the third term is concerned, we observe that, on the set $\{v_n>0\}$, the function $\frac{T_\eps(v_{n,\eps})}{\eps}$ converges a.e.\ to 1, while on the set $\{v_n=0\}$ (where we cannot identify the limit of $\frac{T_\eps(v_{n,\eps})}{\eps}$, but where $\nabla v_n=0$ a.e.\ by Stampacchia's result contained in \cite{S}) the term $H(t, x,v_{n,\eps}+\frac1n,\nabla v_{n,\eps})$ converges a.e.\ to $H(t, x,v_{n}+\frac1n,\nabla v_{n})$, which is zero a.e.\ on this set, since $H(t,x,\frac1n,0)=0$ a.e.\ by assumption \eqref{operH1}.
Therefore, $v_n$ is a weak solution of \eqref{pron3}.
 
\subsection{Basic a priori estimates}

A standard argument allows us to show that some basic estimates on the approximating solutions hold. 
We collect them in the following

\begin{lemma}\label{base}\sl
Let $p \geq 2$, and let $u_{n}$ be a solution to problem \rife{pron}. Then, 
$$
\|u_{n}\|_{L^{p}(0,T;W^{1,p}(\Omega))}\leq C\,,
\qquad
\|u_{n}\|_{L^{\infty}(Q)}\leq C\,,
$$
and
\begin{equation}\label{abs}
\int_{Q}H(t,x,u_n,\nabla u_n)\leq C\,.
\end{equation}
Moreover, there exists a function $u$ in $L^{p}(0,T;W_{0}^{1,p}(\Omega))$ such that (up to subsequences) $u_{n}-\frac1n$ converges to $u$ weakly in $L^{p}(0,T;W_{0}^{1,p}(\Omega))$ and a.e.\ on $Q$. Finally, 
$$
\nabla u_{n} \longrightarrow \nabla u\,, \qquad \mbox{a.e. on $Q$.} 
$$
\end{lemma}
\begin{proof}[{\sl Proof}.]
The proof of the first two estimates is quite standard and can be deduced for instance as in \cite{bg} (see also \cite{bdgo}) using the fact that the lower order term is positive. 

In order to get \rife{abs} one can use $\frac{1}{\eps}\,T_\eps(u_n-\frac1n)$ as test function in \rife{pron}. Integrating by parts, dropping nonnegative terms and letting $\vare$ go to zero one gets, by Fatou's lemma
$$\int_{Q}H(t,x,u_n,\nabla u_n)
\leq 
\int_{\Omega} \Big(u_{0}(x)+\frac1n\Big)\,,
$$
which implies \rife{abs}. The almost everywhere convergence of the gradients of $u_{n}$ is a consequence of \rife{abs} and of a result in \cite{bdgo}. 
 
\end{proof}

\medskip
\section{Proof of Theorem \ref{pm2}: the case $p>2$}\label{3}
In this section we give the proof of Theorem \ref{pm2} in the degenerate case $p>2$. 
\medskip
We wish to pass to the limit in the weak formulation of \rife{pron}. 

A key tool in order to pass to the limit will be the following one.
\begin{lemma}\label{pos}\sl
Let $p > 2$, and let $u_{n}$ be a weak solution of problem \rife{pron}. Then, for any $\omega\subset\subset \Omega$, there exists a constant $c_{\omega}$ such that
$$
u_{n}\geq c_{\omega}>0\,, \quad \mbox{in $(0,T) \times \omega$, for every $n$ in $\mathbb{N}$.}
$$
\end{lemma}

\medskip

Before the proof we recall some technical tools we will use. 
The first one is a well known consequence of Gagliardo-Nirenberg inequality which is valid on any cylinder of the type $Q=(0,T)\times\Omega$ with bounded $\Omega$ (see for instance \cite{n}, Lecture II). 
\begin{lemma}\sl
Let $v\in L^p(0,T; W^{1,p}_0 (\Omega))\cap L^{\infty}(0,T;\elle\beta)$, with $p\geq 1$, $\beta\geq 1$. Then $v\in L^\sigma(Q)$ with $\sigma=p\frac{N+\beta}{N}$ and 
\begin{equation}\label{gnc}
\int_{Q} |v|^\sigma \leq C \|v\|_{L^{\infty}(0,T;\elle\beta)}^{\frac{\beta p}{N}}\int_{Q} |\nabla v |^p\,.
\end{equation}\end{lemma} 
Finally, we will need the following local version of a lemma by Stampacchia (see \cite{S}).

\begin{lemma}\label{locals}\sl
Let $\omega(h,r)$ be a function defined on $[0,+\infty)\times[0,1]$, which is nonincreasing in $h$ and nondecreasing in $r$; suppose that there exist constants $k_{0}\geq 0$, $M$, $\rho$, $\sigma>0$, and $\eta>1$ such that
$$
\omega(h,r)\leq \frac{M \omega(k,R)^{\eta}}{(h-k)^{\rho}(R-r)^{\sigma}}\,,
$$
for all $h>k\geq k_{0}$ and $0\leq r<R\leq 1$. Then, for every $r$ in $(0,1)$, there exists $d>0$, given by
$$
d^{\rho}= \frac{M 2^{\frac{\eta (\rho + \sigma)}{\eta -1}}\omega (k_{0},1)^{\eta -1}}{(1-r)^{\sigma}}
$$
 such that
$$
\omega(d,r)=0.
$$
\end{lemma}

In order to simplify notations, we henceforth write $a(\D u_n)$ instead of $a(t, x, \D u_n)$ and $H(u_n,\D u_n)$ instead of $H(t,x, u_n,\D u_n)$.

\begin{proof}[{\sl Proof of Lemma \ref{pos}}]
We divide the proof into a few steps. 

{\bf Step 1.} There is no loss in generality in assuming that the constant $B$ which appears in \eqref{operH1} satisfies $B>\max(\al,p-1)$. 

We use $v= - u_{n}^{-B}\psi$ in \rife{pronweak}, for any nonnegative $\psi(t,x)\in L^p(0,T;W^{1,p}_0(\Om))\cap L^\infty(Q)$ which is zero in a neighborhood of $(0,T)\times \partial\Om$, in order to obtain
\begin{multline*}
- \int_{0}^{T}\langle u_{n}', u_{n}^{-B}\psi\rangle + B\int_{Q} a(\nabla u_{n})\cdot\nabla u_n u_{n}^{-B-1}\psi+\\ - \int_{Q} a(\nabla u_{n})\cdot\nabla \psi\, u_{n}^{-B} -\int_{Q} H(u_n,\nabla u_n)\,u_{n}^{-B}\psi=0\,,
\end{multline*}
from which, taking into account the assumptions \eqref{oper1}, \eqref{operH1} and $B> \alpha$, one obtains
$$
- \int_{0}^{T}\langle u_{n}', u_{n}^{-B}\psi\rangle - \int_{Q} a(\nabla u_{n})\cdot\nabla \psi\, u_{n}^{-B} \leq 0\,.
$$
Therefore, if we set $u_{n}=\frac{B+1-p}{p-1}w_{n}^{-\frac{p-1}{B+1-p}}$ and $\gamma=\frac{(p-1)(B-1)}{B+1-p}$, we have
\begin{equation}\label{wn3}
\int_{0}^{T}\langle w_{n}', w_{n}^{\ga-1}\psi\rangle - C\,\int_{Q} a(-w_n^{-\frac{B}{B+1-p}}\nabla w_{n})\cdot\nabla \psi\, w_{n}^{\frac{B(p-1)}{B+1-p}} \leq 0\,,
\end{equation}
for every nonnegative $\psi(t,x)\in L^p(0,T;W^{1,p}_0(\Om))\cap L^\infty(Q)$ which is zero in a neighborhood of $(0,T)\times \partial\Om$ and for some positive constant $C$ depending only on $B$ and $p$.
Observe that $B>p-1$ and $p>2$ imply that $\gamma>p-1>1$. Moreover, observe that 
$$
w_{n}(t,x)= c\,n^{\frac{B+1-p}{p-1}} \ \ \text{on} \ (0,T)\times\partial\Omega, \ \ \, w_{n}(0,x)=c\,\Big(u_{0}+\frac1n\Big)^{-\frac{B+1-p}{p-1}}=:w_{0n}\,,
$$
for some positive constant $c$. In particular, since $u_{0}$ is bounded away from zero, then $w_{0n}$ is bounded in $L^{\infty}(\Omega)$ and the values of $w_{n}$ blow up on the boundary as $n$ goes to infinity. We look for an a priori local bound on the $L^{\infty}$ norm of $w_{n}$.

{\bf Step 2.} Local $L^{\infty}$ bound for $w_{n}$.

Without loss of generality we assume that $0\in \Omega$; we will prove that the bound holds true in a ball $B_{\rho}$ centered at zero of radius $\rho$ with $0<r<\rho<R\leq 1$, then a standard covering argument will allow us to conclude.

We fix $k>\|w_{0n}\|_{L^{\infty}(B_{R})}$, and we define the sets
$$
A_{k,\rho}(t)=\{x\in B_{\rho}: w_{n}(t,x)> k\}\,,
\quad
A_{k,\rho}=\{(t,x)\in (0,T)\times B_{\rho}: w_{n}(t,x)> k\}.
$$

We consider a cut-off function $\eta(x)\in C^\infty_{\rm c}(B_R)$ such that 
$$
 0\leq \eta(x)\leq 1\,,\quad
 \eta(x) \equiv 1 \ \hbox{in $B_r$}\,,
 \quad
 |\nabla \eta|\leq \frac{c}{R-r}\,,
$$ 
and use $\psi(t,x) = G_{k}(w_{n}(t,x))\varphi^{\delta}(x)$ as test function in \rife{wn3}, where
$\delta=\frac{p(\gamma+1)}{\gamma-p+1}$\,. Integrating between $0$ and $t<T$, and using the assumptions \eqref{oper1} and \eqref{oper2}, we obtain
\begin{multline*}
\int_{0}^{t}\langle w_{n}',w_{n}^{\gamma-1}G_{k}(w_{n})\varphi^{\delta}\rangle +\int_{0}^{t}\!\!\int_{A_{k,R}(\tau)}|\nabla G_{k}(w_{n})|^{p}\varphi^{\delta}\\ \leq \frac{C}{R-r}\int_{0}^{t}\!\!\int_{A_{k,R}(\tau)}|\nabla G_{k}(w_{n})|^{p-1}\varphi^{\delta-1}G_{k}(w_{n})\,.
\end{multline*}
We use Young's inequality in order to absorb the term $|\nabla G_{k}(w_{n})|^{p-1}$ in the right hand side; using the definition of $\delta$, we get
$$
\int_{0}^{t}\langle w_{n}',w_{n}^{\gamma-1}G_{k}(w_{n})\varphi^{\delta}\rangle
+
\int_{0}^{t}\!\!\int_{A_{k,R}(\tau)}|\nabla G_{k}(w_{n})|^{p}\varphi^{\delta}
\leq
\frac{C}{(R-r)^{p}}\int_{A_{k,R}}G_{k}(w_{n})^{p}\varphi^{\frac{p^{2}}{\gamma-p+1}}\,.
$$
Notice that, since $0\leq \varphi\leq 1$,
\begin{multline*}
 |\nabla \big(G_{k}(w_{n})\varphi^{\delta}\big)|^{p} \leq C\left(|\nabla G_{k}(w_{n})|^{p}\varphi^{\delta p} + \frac{G_{k}(w_{n})^{p}\varphi^{\delta p - p}}{(R-r)^{p}}\right) \\ \leq C\left(|\nabla G_{k}(w_{n})|^{p}\varphi^{\delta } + \frac{G_{k}(w_{n})^{p}\varphi^{\delta - p}}{(R-r)^{p}}\right)\,, 
 \end{multline*}
so that, observing that $\delta - p= \frac{p^{2}}{\gamma-p+1}$, we finally get 
$$
\int_{0}^{t}\langle w_{n}',w_{n}^{\gamma-1}G_{k}(w_{n})\varphi^{\delta}\rangle
+
\int_{0}^{t}\!\!\int_{A_{k,R}(\tau)}|\nabla G_{k}(w_{n})\varphi^{\delta} |^{p}
\leq
\frac{C}{(R-r)^{p}}\int_{A_{k,R}}G_{k}(w_{n})^{p}\varphi^{\frac{p^{2}}{\gamma-p+1}}\,.
$$
Since $\gamma>p-1$, we can use again Young's inequality to have
\begin{equation}\label{plug} 
\begin{split}
\dys \int_{0}^{t}\langle w_{n}',&w_{n}^{\gamma-1}G_{k}(w_{n})\varphi^{\delta}\rangle +\int_{0}^{t}\!\!\int_{A_{k,R}(\tau)}|\nabla \big(G_{k}(w_{n})\varphi^{\delta}\big) |^{p}\\[10pt] 
\dys &\leq \frac{1}{2(\gamma+1) T}\int_{A_{k,R}}\!\!G_{k}(w_{n})^{\gamma+1}\varphi^{\delta}+ \frac{C}{(R-r)^{\frac{(\gamma+1)p}{\gamma+1 - p}}}|A_{k,R}|\\[10pt] 
\dys &\leq \frac{1}{2(\gamma+1)}\ \sup_{t}\int_{A_{k,R}(t)}\!\!G_{k}(w_{n})^{\gamma+1} \varphi^{\delta}\,dx+ \frac{C}{(R-r)^{\frac{(\gamma+1)p}{\gamma+1 - p}}}|A_{k,R}|\,.
\end{split}
\end{equation}
Now we deal with the time derivative part. Let us define, for $s\geq 0$,
$$
\Psi_{k}(s)=\int_{0}^{s}G_{k}(\sigma)\sigma^{\gamma-1}\ d\sigma.
$$
Then it is easy to check that 
$$
 \Psi_{k}(s)\geq \frac{1}{\ga+1}G_{k}(s)^{\gamma+1}\,.
$$
that $k>\|w_{0}\|_{L^{\infty}(B_{R})}$, we obtain
$$
 \int_{0}^{t}\langle w_{n}',w_{n}^{\gamma -1}G_{k}(w_{n})\varphi^{\delta}\rangle =\int_{B_{R}}\Psi_k(w_{n}(t,x))\varphi^{\delta}\geq \frac{1}{\gamma+1}\int_{B_{R}}G_{k}(w_{n})^{\gamma+1}\varphi^{\delta}. 
$$
Therefore we can use \rife{plug} in order to deduce 
$$
\frac{1}{\gamma+1}\int_{B_{R}}G_{k}(w_{n})^{\gamma+1}\varphi^{\delta}\leq\frac{1}{2(\gamma+1)}\sup_{t}\int_{A_{k,R}}G_{k}(w_{n})^{\gamma+1} \varphi^{\delta}+ \frac{C}{(R-r)^{\frac{(\gamma+1)p}{\gamma+1 - p}}}|A_{k,R}|
$$
and we can take the supremum over $t\in (0,T)$ on the left in order to get 
$$
\sup_{t}\int_{A_{k,R}(t)}G_{k}(w_{n})^{\gamma+1} \varphi^{\delta}\leq \frac{C}{(R-r)^{\frac{(\gamma+1)p}{\gamma+1 - p}}}\,|A_{k,R}|. 
$$
Gathering together all these facts, and using again that $\varphi^{\delta}\geq \varphi^{\delta(\gamma+1)}$, we end up with the following estimate
$$
\dys \sup_{t}\int_{A_{k,R}(t)}\!\!\big(G_{k}(w_{n}) \varphi^{\delta}\big)^{\gamma+1}+\int_{A_{k,R}}\!\!|\nabla \big(G_{k}(w_{n})\varphi^{\delta}\big) |^{p} \dys\leq \frac{C}{(R-r)^{\frac{(\gamma+1)p}{\gamma+1 - p}}}|A_{k,R}|\,.
$$
We are now in the position to apply the Gagliardo-Nirenberg inequality \rife{gnc} to the function $G_{k}(w_{n}) \varphi^{\delta}$, with $\beta=\gamma+1$; recalling that $\varphi\equiv 1$ on $B_{r}$, we obtain
$$
\int_{A_{k,r}} G_{k}(w_{n})^{p\frac{N+\gamma+1}{N}}\leq \frac{C}{(R-r)^{\frac{(\gamma+1)p(N+p)}{(\gamma+1 - p)N}}}|A_{k,R}|^{1+
\frac{p}{N}}. 
$$
Stampacchia's procedure is now quite standard. For $h>k$, one obtains
$$
\int_{A_{k,r}} G_{k}(w_{n})^{p\frac{N+\gamma+1}{N}}\geq \int_{A_{h,r}} G_{h}(w_{n})^{p\frac{N+\gamma+1}{N}}\geq (h-k)^{p\frac{N+\gamma+1}{N}}|A_{h,r}|\,,
$$
that is,
$$
|A_{h,r}|\leq \frac{C \dys |A_{k,R}|^{1+\frac{p}{N}}}{(h-k)^{p\frac{N+\gamma+1}{N}}(R-r)^{\frac{p(\gamma+1)(N+p)}{(\gamma+1 - p)N}}}. 
$$
Therefore, if we choose $\omega(h,r)=|A_{h,r}|$, we can apply Lemma \ref{locals} in order to deduce that, for every fixed $\rho\in (r,R)$, 
$|A_{h,r}|=0$ if $h$ is larger than some constant $C_\rho$. It follows that 
\begin{equation}\label{boundw}
w_{n}\leq C_{\rho}, \ \ \ \text{a.e. on $(0,T)\times B_{\rho}$,\quad for every $n$. }
\end{equation}

{\bf Step 3.} End of the proof.
Recalling that $B>p-1$ and the definition of $w_{n}$ we use \rife{boundw} to have, a.e. on $(0,T)\times B_{\rho}$
$$
u_{n}= \frac{B+1-p}{p-1} w_{n}^{-\frac{p-1}{B+1-p}}\geq \frac{B+1-p}{p-1} C_{\rho}^{-\frac{p-1}{B+1-p}}\doteq c_{\rho}>0. 
$$
As we said, by means of a standard covering procedure this estimate can be proven to hold on any set of the form $(0,T) \times \omega$, with $\omega \subset\subset \Omega$.
\end{proof}

\medskip

\subsection{Passing to the limit}

In order to pass to the limit as $n$ tends to infinity, we need the following result.
\begin{proposition}\label{str}\sl
We have
$$
u_{n}\longrightarrow u, \qquad \mbox{strongly in $L^{p}(0,T;W^{1,p}(\omega))$,}
$$
for every open set $\omega\subset\subset\Omega$.
\end{proposition}

\begin{proof}[{\sl Proof}.]

The sequence $\{u_n'\}$ is bounded in
$L^{p'}(0,T;W^{-1,p'}(\Omega))+L^1(Q)$. Using the Aubin-Simon
compactness argument (see Corollary~4 in \cite{S}) we deduce that, up to a subsequence,
$$
u_n\rightarrow u \mbox{ in } L^p(Q)\,,
$$
for some $u$ in $L^{p}(0,T;W^{1,p}_{0}(\Omega)) \cap L^{\infty}(Q)$.
We will prove that, for every open set $\omega\subset\subset\Omega$,
\begin{equation}
 \label{convi}
u_n \rightarrow u \quad \mbox{in } L^p(0,T; W^{1,p}(\omega)).
\end{equation}

We now introduce a 
classical regularization $u_{\nu}$ of the 
function $u$
with respect to time (see \cite{LA}). For
every $\nu\in \enne$, we define $u_\nu$ as the solution of the
Cauchy problem
$$
\left\{
\begin{array}{cl}
\dys {\frac1\nu}u_\nu'+u_\nu = u\,,\\
\\
u_\nu(0)=u_{0,\nu}\,,\\
\end{array}
\right.
$$
where $u_{0,\nu}$ belongs to $W^{1,p}_{0}(\Om) \cap L^{\infty}(\Om)$
and satisfies
$$
 u_{0, \nu}\to u_0 \quad\text{strongly in
 $L^1(\Omega)$ and $*$-weakly in $L^\infty(\Omega)$ ,}\quad
 \lim_{\nu\to +\infty}\frac{1}{\nu}\Vert u_{0,\nu}\Vert_{W^{1,p}_{0}(\Om)}
 =0\,.
$$
Then one has (see
\cite{LA}):
$$
 u_\nu \in L^p(0,T;W^{1,p}_0(\Om))\,\qquad
 u_\nu' \in L^p(0,T;W^{1,p}_0(\Om))\,,
$$
$$ \Vert
u_\nu\Vert_{L^\infty(Q)}\le\Vert u\Vert_{L^\infty(Q)}\,,
$$
and, as $\nu$ tends to infinity,
\begin{equation}\label{strongunu}
u_\nu \to u\quad \hbox{strongly in
}L^p(0,T;W^{1,p}_0(\Om)).
\end{equation}

\noindent Let $\vp_{\la } (s)= s {\rm e}^{\la s^2}$ (with $\la$ to be chosen later). We will denote by $\eps(\nu,n)$ any 
quantity such that
\begin{equation*}
\begin{array}{l}
\dys \lim_{\nu\to+\infty}
\limsup_{n\to +\infty}|\eps(\nu,n)|=0.
\end{array}
\end{equation*}
\noindent For $0\leq \phi \in C^{\infty}_{\rm c} (\Omega)$ we have that
\begin{equation}
 \label{dus}
\int_0^T \langle u_n',\varphi_\lambda(u_n-u_\nu)\phi
\rangle\geq\eps(\nu,n)\,,
\end{equation}
(see \cite{do,lp}).

\medskip

Now, using \eqref{dus} and $\vp_{\la} (u_n-u_\nu) \phi $ as test
function in \eqref{pronweak}, we obtain
\begin{multline*}
\dys \int_{Q} a(\nabla u_n)\cdot \D (u_n-u_\nu) \vp_{\la}
'(u_n-u_\nu) \phi +\dys\int_{Q} a(\nabla u_n) \cdot \D \phi\,
\vp_{\la} (u_n-u_\nu)
\dys \\
+\dys\int_{Q} H(u_n,\D u_n) \vp_{\la} (u_n-u_\nu) \phi \leq -\eps(\nu,n).
\end{multline*}

Moreover, if $\omega\subset\subset \Omega$ is such that
$\mbox{supp\,} \phi \subset\omega$, since $u_n \to u$ weakly in
$L^p(0,T;W^{1,p}_{0} (\Omega))$ and $\varphi_\lambda(u_n-u_\nu)$ converges
to $\varphi_\lambda(u-u_\nu)$ $\ast$-weakly in
$L^{\infty}(Q)$, we have
$$
 \int_{Q} a(\nabla u_n)
\cdot \D \phi \, \vp_{\la} (u_n-u_\nu)\, = \eps (\nu,n)\,.
$$
If $c_{\omega}$ is the constant given by Lemma \ref{pos}, we have, recalling that ${\rm supp}\, \phi \subset\omega$,
\begin{multline*}
\dys\left|\int_{Q} H(u_n,\D u_n) \vp_{\la} (u_n-u_\nu) \phi\right| \dys \leq
B\,\int_{\omega\times(0,T)} \frac{|\nabla u_n|^p}{u_{n}} |\vp_{\la} (u_n-u_\nu) |\phi \\ \dys\leq 
\frac{B}{c_{\omega}}\int_ {Q} |\nabla u_n|^p |\vp_{\la} (u_n-u_\nu) |\phi
\, .
\end{multline*}
Thus,
\begin{equation} \label{ccua} 
\int_{Q} a(\nabla u_n) \cdot \D (u_n-u_\nu) \vp_{\la}'(u_n-u_\nu) \phi
-
\frac{B}{c_{\omega}} \int_{Q} |\nabla u_n|^p |\vp_{\la} (u_n-u_\nu) |\phi
\leq
\eps(\nu,n)\,.
\end{equation}
Then we can write
\begin{multline*}
\int_Q a(\nabla u_n) \cdot \D (u_n-u_\nu) \vp_{\la}'(u_n-u_\nu) \phi
=
\int_Q \big[a(\nabla u_n) - a(\nabla u_\nu)\big] \cdot \D (u_n-u_\nu) \vp_{\la}'(u_n-u_\nu) \phi
\\
+
\int_Q a(\nabla u_\nu) \cdot \D (u_n-u_\nu) \vp_{\la}'(u_n-u_\nu) \phi
\\
=
\int_Q \big[a(\nabla u_n) - a(\nabla u_\nu)\big] \cdot \D (u_n-u_\nu) \vp_{\la}'(u_n-u_\nu) \phi 
+
\eps(\nu,n)\,.
\end{multline*}
Similarly,
\begin{multline*}
\dys\int_{Q} |\nabla u_n|^p |\vp_{\la} (u_n-u_\nu) |\,\phi
\leq \al^{-1}
\int_{Q} a(\D u_n)\cdot \D u_n\, |\vp_{\la} (u_n-u_\nu) |\,\phi 
\\
=
\al^{-1}
\int_{Q} \big[a(\D u_n)- a(\D u_\nu)\big]\cdot \D (u_n-u_\nu)\, |\vp_{\la} (u_n-u_\nu) |\,\phi
\\
+
\int_{Q} a(\D u_\nu)\cdot \D (u_n-u_\nu)\, |\vp_{\la} (u_n-u_\nu) |\,\phi
+ 
\int_{Q} a(\D u_n)\cdot \D u_\nu\, |\vp_{\la} (u_n-u_\nu) |\,\phi
\\
=
\al^{-1}
\int_{Q} \big[a(\D u_n)- a(\D u_\nu)\big]\cdot \D (u_n-u_\nu)\, |\vp_{\la} (u_n-u_\nu) |\,\phi
+\eps(\nu,n)\,.
\end{multline*}
Therefore, from \eqref{ccua} we obtain
\begin{multline*}
\int_{Q} \big[a(\D u_n)- a(\D u_\nu)\big]\cdot \D (u_n-u_\nu)\, \Big[\vp'_{\la} (u_n-u_\nu) -\frac{B}{\al c_{\omega}}|\vp_{\la} (u_n-u_\nu) |\Big]\,\phi\leq \eps(\nu,n)\,.
\end{multline*}
Choosing $\la$ large enough so that $ \dys
\vp_{\la }' (s)- \frac{B}{\al c_{\omega}} |\vp_{\la } (s)| \geq \frac{1}{2}$ for every
$s\in\mathbb{R}$, we deduce that
$$
\int_{Q} \big[a(\D u_n)- a(\D u_\nu)\big]\cdot \D (u_n-u_\nu)\,\phi\leq \eps(\nu,n)\,.
$$
From here it is standard (see for example \cite{bmp}) to prove that $u_{n} - u_{\nu}$ tends to zero strongly in $L^{p}(0,T;W^{1,p}(\omega))$. Recalling \eqref{strongunu}, we thus have that \eqref{convi} holds.
\end{proof}

Using Proposition \ref{str} we can prove the (local) strong convergence of the lower order terms.
\begin{lemma}\label{strlower}\sl
We have
$$
H(u_{n},\D u_{n}) \to H(u,\D u)\,,
\qquad
\mbox{locally strongly in $L^{1}(Q)$.}
$$
\end{lemma}
\begin{proof}[{\sl Proof}.]
Gathering together the results of Theorem \ref{str}, Lemma \ref{base}, and Lemma \ref{pos}, we can apply Lebesgue's dominated convergence theorem to prove that
$$
\frac{|\D u_n|^p}{u_n} \to \frac{|\D u|^p}{u}\,,
\qquad
\mbox{locally strongly in $L^{1}(Q)$.}
$$
It is then straightforward to conclude using \eqref{operH1} and Vitali's theorem.
\end{proof}

Thanks to all the results proved so far we can pass to the limit in \rife{pron}, to have that $u$ is a solution of \eqref{pro} in the sense of Definition \ref{defin}, thus concluding the proof of Theorem \ref{pm2} in the degenerate case $p > 2$.

\begin{remark}\rm
We want to point out that since the same argument of the proof of Lemma \ref{pos} applies to supersolutions of problem \ref{pron}. Thus, in particular, one could prove the same result of Theorem \ref{pm2} also for nonhomogeneous problems like
$$
\left\{
\begin{array}{cl}
\dys 
u' - {\rm div} \, a(t,x,\D u) + H(t,x,u,\D u) = f &\mbox{in $Q$,} \\[1.5 ex]
\dys u(0,x)= u_0 (x) &\mbox{in $\Omega$,} \\[1.5 ex]
\dys u(t,x)=0 &\mbox{on $\Gamma$,}
\end{array}\right.
$$ 
with $0 \leq f$ in $L^{r}(0,T;L^{q}(\Omega))$, with
$$
\frac{p}{r} + \frac{N}{q} < p\,,
\quad
r \geq p'\,,
\quad
q > 1\,.
$$
\end{remark}

\section{Proof of Theorem \ref{pm2}: the case $p = 2$}\label{4}

As before, we start from the approximate problems \eqref{pron}, and the key result will be the strong positivity of $u_{n}$ on $(0,T) \times \omega$, with $\omega \subset\subset \Omega$. 
\begin{lemma}\label{pos2}\sl
Assume that \eqref{oper1}, \eqref{oper2}, \eqref{oper3} and \eqref{operH1} hold with $p=2$ and $B<\alpha$, and let $u_{n}$ be a weak solution of problem \rife{pron}. Then, for any $\omega\subset\subset \Omega$, there exists a constant $c_{\omega}$ such that
$$
u_{n}\geq c_{\omega}>0\,, \quad \mbox{in $(0,T) \times \omega$, for every $n$ in $\mathbb{N}$.}
$$
\end{lemma}
\begin{proof}[{\sl Proof}.]
We multiply the equation in \eqref{pron} by $u_n^{-\theta}\,\psi$, where 
$$
\theta = \frac{B}\alpha \in (0,1)
$$
and $\psi(t,x)\in L^p(0,T;W^{1,p}_0(\Om))\cap L^\infty(Q)$ is a nonnegative function which is zero in a neighborhood of $(0,T)\times \partial\Om$.
As in Section \ref{3}, a straightforward calculation shows that $u_{n}$ satisfies
$$
\int_{0}^{T}\langle u_{n}', u_{n}^{-\theta}\psi\rangle + \int_{Q} a(\nabla u_{n})\cdot\nabla \psi\, u_{n}^{-\theta} \geq 0\,.
$$
Then, if we define
$$
 v_n = \left(\frac{u_n}{1-\theta}\right)^{1-\theta}\,,
$$
the previous inequality can be read as
$$
 (1-\theta)^{\theta+2} \int_{0}^{T}\langle v_{n}', \psi\rangle + \int_{Q} \tilde a(t,x,v_n,\nabla v_{n})\cdot\nabla \psi \geq 0\,,
$$
where we have set
$$
 \tilde a(t,x,s,\xi) = s^{-\frac{\theta}{1-\theta}}\,a(t,x, s^{\frac{\theta}{1-\theta}}\,\xi)\,.
$$
In other words, $v_n$ is a weak solution of the variational inequality
$$
\left\{
\begin{array}{cl}
\dys (1-\theta)^{\theta+2}v_n' -{\rm div}\,\tilde a(t,x,v_n,\nabla v_{n}) \geq 0 &\mbox{in $Q$,} \\[1.5 ex]
v_n(0,x)= \left(\frac{u_0+\frac1n}{1-\theta}\right)^{1-\theta} &\mbox{in $\Omega$,} \\[1.5 ex]
v_n(t,x)=\left(\frac{1}{(1-\theta)n}\right)^{1-\theta} &\mbox{on $\Gamma$.} 
\end{array}\right.
$$
Since the vector-valued function $\tilde a$ satisfies the standard inequalities
\begin{gather*}
 \tilde a(t,x,s,\xi)\cdot \xi \geq \al |\xi|^2\,,\\
 |\tilde a(t,x,s,\xi)| \leq \beta |\xi|\,,\\
\big(\tilde a(t,x,s, \xi)- \tilde a(t,x,s, \eta)\big)\cdot (\xi-\eta)>0 \,,
\end{gather*}
by the standard maximum principle (see \cite{LSU}) we easily deduce that $v_n \geq C_{\omega}>0$, in any $\omega\subset\subset\Omega$; this implies the result. 
\end{proof}

Once we have proved Lemma \ref{pos2}, the proof of Theorem \ref{pm2} can be concluded as in Section~\ref{3}, using Proposition \ref{str} and Lemma \ref{strlower}, which continue to hold also for $p = 2$, with the same proof.

\section{The case $p=2$ and $B\geq \alpha$}\label{5}

In this section we try to explain how nonexistence of solutions (in the sense of approximating sequences) may actually occur for problem \rife{pro} with $p=2$ if the absorption parameter $B$ is too large. For the sake of explanation, we will restrict our attention to the model problem
$$
\left\{
\begin{array}{cl}
\dys u'- \Delta u + B\frac{|\D u|^2}{u} = 0  &\mbox{in $Q$,} \\[1.5 ex]
u(0,x)= u_0 (x) &\mbox{in $\Omega$,} \\[1.5 ex]
u(t,x)=0 &\mbox{on $\Gamma$,} 
\end{array}\right.
$$
for $B\geq 1$. An easy rescaling argument will then allow to deal with the case $B \geq \alpha$ for
$$
u' - \alpha\Delta u + B\frac{|\D u|^{2}}{u} = 0\,.
$$

In order to provide a simpler proof, without loss of generality, we will consider a constant initial datum (e.g., $u_{0}=1$); the case of a general initial datum can be obtained in a straightforward way by comparison (recall that $u_{0}\geq c>0$) and elementary changes in the proof. 

Our result is the following.
\begin{proposition}\label{none}\sl
Let $B \geq 1$, and let $u_{n}$ be a solution of problem 
\begin{equation}\label{exe}
\left\{
\begin{array}{cl}
\dys u_{n}'- \Delta u_n + B\frac{|\D u_n|^2}{u_n} = 0  &\mbox{in $Q$,} \\[1.5 ex]
u_n(0,x)= 1 &\mbox{in $\Omega$,} \\[1.5 ex]
u_n(t,x)=\frac1n &\mbox{on $\Gamma$.} 
\end{array}\right.
\end{equation}
Then 
$$
u_n\longrightarrow 0\ \ \text{uniformly on compact subsets of $Q$}\,.
$$
\end{proposition}
\begin{proof}[{\sl Proof}.]
We first analyze the case $B=1$, readapting an idea of \cite{mp}. 
Let $u_{n}$ be a solution to \rife{exe} and define $v_{n}=-\log (u_{n})$. Then $v_{n}$ solves 
$$
\left\{
\begin{array}{cl}
\dys v_n'- \Delta_{} v_n = 0  &\mbox{in $Q$,} \\[1.5 ex]
v_n(0,x)= 0 &\mbox{in $\Omega$,} \\[1.5 ex]
v_n(t,x)=\log(n) &\mbox{on $\Gamma$.} 
\end{array}\right.
$$
On the other hand it is clear that $v_{n}=(1-w)\log(n)$, where $w$ is the unique solution to problem 
$$
\left\{
\begin{array}{cl}
\dys w'= \Delta_{} w  &\mbox{in $Q$,} \\[1.5 ex]
w(0,x)= 1 &\mbox{in $\Omega$,} \\[1.5 ex]
w(t,x)=0 &\mbox{on $\Gamma$.} 
\end{array}\right.
$$
By the strong maximum principle we have 
$$
0<w(t,x)<1\ \ \text{for any $(t,x)$ in $Q$}, 
$$
so that $v_{n}\longrightarrow +\infty$ uniformly on compact subsets of $Q$. Hence,
$$
u_{n}={\rm e}^{-v_{n}}\longrightarrow 0 \ \ \text{uniformly on compact subsets of $Q$,}
$$
as desired.

If $B>1$, a similar argument applies; here the right change of variable is
$$
v_{n}=-\frac{u_{n}^{1-B}}{1-B}\,, 
$$
which transforms problem \rife{exe} into
$$
\left\{
\begin{array}{cl}
\dys v_n'- \Delta_{} v_n = 0  &\mbox{in $Q$,} \\[1.5 ex]
v_n(0,x)= \frac{1}{B-1} &\mbox{in $\Omega$,} \\[1.5 ex]
v_n(t,x)= \frac{n^{B-1}}{B-1} &\mbox{on $\Gamma$.} 
\end{array}\right.
$$
The conclusion then follows as before, since $v_{n}$ locally uniformly on compact subsets of $Q$, and $u_{n}=[(B-1)v_{n}]^{-\frac{1}{B-1}}$.

\end{proof}

\medskip
\section{The case $p<2$. Some partial results.}\label{6}

In this last section we will briefly discuss the singular case $p<2$. We will show that no solutions in the sense of Definition \ref{defin} can be expected as finite time extinction may occur. Moreover, we will also show some explicit nondegenerate examples of evolution. 
\subsection{Extinction in finite time.}
Consider
\begin{equation}\label{proETF}
\left\{
\begin{array}{cl}
\dys u'- \Delta_{p} u + B\frac{|\D u|^p}{u} = 0  &\mbox{in $Q$,} \\[1.5 ex]
u(0,x)= u_0 (x) &\mbox{in $\Omega$,} \\[1.5 ex]
u(t,x)=0 &\mbox{on $\Gamma$,} 
\end{array}\right.
\end{equation}
 where $T>0$, $N\geq 2$. Concerning the initial datum $u_{0}$ we assume that $u_{0}$ is a nonnegative bounded function such that $u_{0}\geq c>0$ almost everywhere in $\Omega$.

We have the following result that should be compared with finite time extinction property for the, formally equivalent, doubly nonlinear case (see for instance \cite{wz} and references therein). 

\begin{lemma}\sl
Let $p<2$, $B<p-1$, $u_{0}$ be as above, and let $u$ be a solution of \eqref{proETF}; then there exists $T^{*} > 0$ such that $u(T^{*},x) \equiv 0$.
\end{lemma}
\begin{proof}[{\sl Proof}.]
Let $$ w= \frac{p-1}{p-1-B}u^{\frac{p-1-B}{p-1}}\,,
$$
so that $w$ solves
$$
\left\{
\begin{array}{cl}
\dys w^{\beta}w'- \Delta_{p} w = 0  &\mbox{in $Q$,} \\[1.5 ex]
w(0,x)= \frac{p-1}{p-1-B}u_{0}(x)^{\frac{p-1-B}{p-1}} &\mbox{in $\Omega$,} \\[1.5 ex]
w(t,x)=0 &\mbox{on $\Gamma$,} 
\end{array}\right.
$$
with $\beta=\frac{(2-p)B}{p-1-B}$. We consider $w^{\gamma}$ (where $\gamma$ will be chosen later) as test function in the previous problem, and we integrate on $Q_{t} = (0,t) \times \Omega$, with $0 < t < T$. We obtain 
$$
\int_{\Omega} w(t)^{\gamma+\beta+1} +\int_{Q_{t}}|\nabla w|^{p} w^{\gamma-1}\leq C,
$$
that implies
$$
\int_{\Omega} w(t)^{\gamma+\beta+1} +\int_{Q_{t}}|\nabla w^{\frac{\gamma+p-1}{p}}|^{p} \leq C.
$$
Then using Sobolev inequality
$$
\int_{\Omega} w(t)^{\gamma+\beta+1} +\left(\int_{Q_{t}} w^{\frac{(\gamma+p-1)p^{\ast}}{p}}\right)^{\frac{p}{p^{\ast}}} \leq C.
$$ 
Now observe that, by H\"older inequality, that can be used as $(\gamma +p-1)\frac{p^{\ast}}{p}>\gamma+\beta+1$ (for $\gamma$ large enough), 
$$
\int_{Q_{t}} w^{\gamma+\beta+1} \leq C \left(\int_{Q_{t}} w^{\frac{(\gamma+p-1)p^{\ast}}{p}}\right)^{\frac{p(\gamma +\beta+1)}{p^{\ast}(\gamma +p-1)}}\,.
$$
Hence, if we define 
$$
\eta(t)=\int_{Q_{t}} w^{\gamma+\beta+1}\,,
$$
we have, for some positive constants $C_{1}$ and $C$, 
$$
\eta'(t)+C_{1}\eta(t)^{\frac{\gamma+p-1}{\gamma+\beta+1}}\leq C\,.
$$
We now remark that $\frac{\gamma+p-1}{\gamma+\beta+1}<1$ due to the assumptions $p < 2$ and $B < p - 1$. Therefore, by standard ODE analysis, there exists $T^{*} > 0$ such that $\eta(T^{*}) = 0$; this implies the finite time extinction for $w$, hence for $u$. 
\end{proof}

\subsection{An example of nontrivial evolution.} As we already said, the analysis in the case $p<2$ is much more delicate. In the previous section we showed that in some cases finite time extinction may occur and so no solutions (in the sense of Definition \ref{defin}) can be found in general. 
In this section we provide an explicit example of solution that shows that, in some cases, a solution exists at least for small times. 

\begin{example}
Let $\frac{2N}{N+2}<p<2$ and let $B$ be such that 
$$
0<B<\frac{(p^{\ast}-2)(p-1)}{p^{\ast}-p}<p-1.
$$
First of all, consider the solution $v$ to problem
\begin{equation}\label{vu}
\left\{
\begin{array}{cl}
\dys -\Delta_{p} v=\gamma^{p-1}v^{\frac{(\gamma-1)(p-1)+1}{\gamma}}  &\mbox{in $\Omega$,} \\[1.5 ex]
v= 0 &\mbox{on $\partial\Omega$,}
\end{array}\right.
\end{equation}
where, 
$$
\gamma=\frac{p-1-B}{p-1}>0.
$$
Problem \rife{vu} admits a nontrivial, positive, variational solution obtained by a standard mountain pass procedure as soon as
$$
p-1 < \frac{(\gamma-1)(p-1)+1}{\gamma}<p^{\ast} -1,
$$ 
and an easy calculation shows that the previous relation is equivalent to
$$
B<\frac{(p^{\ast}-2)(p-1)}{p^{\ast}-p}\,.
$$

Consider now the function $\phi\doteq v^{\frac{1}{\gamma}}$, which solves the problem
$$
\left\{
\begin{array}{cl}
\dys -\Delta_{p} \phi+B\frac{|\nabla \phi|^{p}}{\phi}=\phi &\mbox{in $\Omega$,} \\[1.5 ex]
\phi= 0 &\mbox{on $\partial\Omega$,}
\end{array}\right.
$$
and define 
$$
u(t,x)=(1-(2-p)t)^{\frac{1}{2-p}}\phi(x)\,. 
$$
We can compute the problem solved by $u$, we have

$$
u'- \Delta_{p} u + B\frac{|\D u|^p}{u} 
= (1-(2-p)t)^{\frac{1}{2-p}}\left(-\Delta_{p} \phi+B\frac{|\nabla \phi|^{p}}{\phi}-\phi\right)=0\,,
$$
so that $u$ is a solution of problem
$$
\left\{
\begin{array}{cl}
\dys u'- \Delta_{p} u + B\frac{|\D u|^p}{u} = 0 &\mbox{in $Q$,} \\[1.5 ex]
u(0,x)= \phi (x) &\mbox{in $\Omega$,} \\[1.5 ex]
u(t,x)=0 &\mbox{on $\Gamma$,} 
\end{array}\right.
$$
that has $T^{*}=\frac{1}{2-p} > 0$ as extinction time. 
\end{example}

\end{document}